\documentclass[a4paper,11pt]{amsart}
\usepackage{amssymb}
\usepackage{amsmath}
\usepackage{amsthm}
\usepackage{amscd}
\usepackage[colorlinks=true,linkcolor=blue,urlcolor=black]{hyperref}\usepackage{float}
\usepackage[utf8]{inputenc}

\usepackage{stmaryrd}

\usepackage{amsfonts}
\usepackage{latexsym}
\usepackage[all]{xy}
\usepackage{graphicx}
\usepackage{pb-diagram}
\usepackage{verbatim}
\usepackage{anysize}
\usepackage{cancel}
\usepackage{color}
\usepackage{xcolor}
\usepackage[normalem]{ulem}

\usepackage{tikz,tikz-cd}
\usetikzlibrary{decorations.pathmorphing,calc,shadows.blur,shadings}

\usepackage[cal=boondox,calscaled=.96]{mathalfa}

\usepackage{fancyhdr}

\marginsize{2cm}{2cm}{2.5cm}{2.5cm}

\def\QQ{\mathbb{Q}}
\def\NN{\mathbb{N}}

\def\FF{\mathbb{F}}

\def\lbr{\llbracket}
\def\rbr{\rrbracket}

\newcommand{\N}{\mathbb N}

\newtheorem{teo}{Theorem}[section]
\newtheorem{prop}[teo]{Proposition}
\newtheorem{cor}[teo]{Corollary}
\newtheorem{lem}[teo]{Lemma}

\newtheorem{obs}[teo]{Remark}
\newtheorem{ej}[teo]{Example}

\newtheorem{Def-Lem}[teo]{Definition-Lemma}

\theoremstyle{definition}
	\newtheorem{Def}[teo]{Definition}

\DeclareMathOperator{\Hom}{\rm Hom}

\DeclareMathOperator{\Leap}{\rm Leaps}

\DeclareMathOperator{\rank}{\rm rank}

\DeclareMathOperator{\End}{\rm End}

\DeclareMathOperator{\IHS}{\rm IHS}

\DeclareMathOperator{\IDer}{\rm IDer}

\DeclareMathOperator{\Der}{\rm Der}

\DeclareMathOperator{\HS}{\rm HS}

\DeclareMathOperator{\Id}{\rm Id}

\DeclareMathOperator{\Min}{\rm Min}

\DeclareMathOperator{\het}{\rm ht}

	\title{Integrability in the sense of Hasse-Schmidt and $p^e$-basis}
	\author{
		 Mar\'ia de la Paz Tirado Hern\'andez}

	\thanks{
	The author was partially supported by Spanish Ministry of Science and Innovation, [PID2024-156912NB-I00].
	}


	\AtEndDocument{\bigskip{\footnotesize
			
			\noindent\textsc{Depto. Matem\'aticas, Facultad de Ciencias, Universidad Aut\'onoma de Madrid and Instituto de Ciencias Matem\'aticas, ICMAT,  CSIC-UAM-UC3M-UCM, Cantoblanco 28049 Madrid, Spain.} \newline
			\textit{E-mail address}: \texttt{maria.tirado@uam.es}	}}
	
\begin{document}
	
	\maketitle 
		\begin{abstract} Let $R$ be a ring of positive characteristic with a $p$-basis over a subring $k$.  We explore necessary and sufficient conditions under which the module of $k$-derivations of $R$ coincides with the module of integrable $k$-derivations of $R$.
	\end{abstract}
	\section{Introduction}
		
	Let $k$ be a ring and let $R$ be a commutative $k$-algebra   with unit. The $R$-module of derivations $
	\Der_k(R)$ codifies information about the infinitesimal structure of the ring $R$ relative to $k$. For instance, when $k$ is a perfect field and $R$ is essentially of finite type over $k$, then $R$ is regular if and only if $
	\Der_k(R)$ is a locally free $R$-module of rank equal to the Krull dimension of $R$.

	\medskip
	
	Quite generally,    if $k$ is a field containing ${\mathbb Q}$, we have that   $\rank 
	\Der_k(R)\leq \dim R$  (see \cite{Matsumura}, \cite{Molinelli}). This last inequality might  not hold   when $k$ is a field of positive characteristic. In fact,  this  is just an example of how wildly $\Der_k(R)$ can behave when working with rings that  contain a  field  of positive characteristic. In this last case, the $R$-submodule of the {\em integrable derivations}, $\IDer_k(R) \subseteq \Der_k(R)$  is an object    with better behavior. Recall that  $\IDer_k(R) \subseteq \Der_k(R)$   consists on the derivations  that can be extended to a  Hasse-Schmidt derivation  of $R$ over $k$ (of lenght $\infty$). We refer the reader to   \cite{BK}, \cite{mat-intder-I}, \cite{Mo} or  \cite{Se} for examples in which results for $\Der_k(R)$ when ${\mathbb Q}\subset k$ that are no longer valid if $k$ has prime characteristic,  are extended to  positive characteristic for $\IDer_k(R)$ instead. 
	
	\medskip
	
	It is quite natural to wonder how far from one another are the two modules of derivations, $\IDer_k(R)$ and $\Der_k(R)$. For instance, in \cite{BTH} the following result is proven: 
	\begin{teo}
		Let $k$ be a regular ring,  set $A=k[x_1,\ldots, x_n]$ and let  $I\subset A$ be a radical ideal with  $r=\max \{\het(P)\ | \ P\in \Min(I)\}$.  Let  $R=A/I$ and  let  $J_r$ be the $(n-r)$--Fitting ideal  of $\Omega_{R/k}$.  Then,   $J_r\Der_k(R)\subset \IDer_k(R)$.
	\end{teo} 
	In particular, the theorem  gives a procedure to finding integrable derivations once a derivation of $R$ over $k$ is known. However,  there are situations in which the theorem does not tell us anything: 
	
	\begin{ej}\cite[Counterexample 3.9]{Ti2}\label{ejemplo_1} Let $s, t$ be  variables, let $k={\mathbb F}_2(s,t)$ and let $R=k[x,y]/\langle x^2+y^2+tx^4+sy^4\rangle$. Observe that the $R$-module $\Der_k(R)=\langle \partial_x, \partial_y\rangle$ has rank 2 while the Krull dimension of $R$ is 1. Here the corresponding Fitting ideal of $\Omega_{R/k}$ is $J_1=(0)$. On the other hand, it can be checked that $\IDer_k(R)=0$. 
	\end{ej}
	
	The results in this paper are in part motivated by trying to understand examples as Example \ref{ejemplo_1}.  In particular,   for a given ring, when can we expect to have the equality $\IDer_k(R)=\Der_k(R)$?  This question is connected to the following  formulated in  \cite{FernandezNarvaez}, where the authors mention    that they do not know any example of a regular local ring containing a quasi-coefficient field $k$ in which 
	$\Der_{k}(R)\neq \IDer_{k}(R)$. 
	
	\medskip

	We will be addressing our questions under the assumption of the existence of a $p$-basis. More precisely, we generalize \cite[Theorem 2]{Fur80}, which states that if $R$ is a reduced ring of characteristic $p$ with a $p$-basis over $R^p$, then $\Der(R)=\IDer(R)$. Namely, using the notion of $p^e$-basis, which will be dicussed in section \ref{Sec-MainTheorem}, we prove the following theorem:
	
	\medskip 
	\noindent{\bf Theorem \ref{teo-Equivalencias}} {\em 
	Let $R$ be a ring of prime characteristic $p>0$ and $e\geq 1$ be a positive integer. Assume that $R$ has a $p$-basis over a subring $k$. Then the following properties are equivalent:

	1.  $R$ has a $p^e$-basis over $k$.	
	
	2.	$R$ has a $p^s$-basis over $k$ for all $s \leq e$.
		
	3. $\Der_k(R) = \IDer_k(R; p^{e}-1)$. 
		 }

	\medskip
	
	Continuing with Example \ref{ejemplo_1}, the set $\{x,y\}$ is a $2$-basis of $R$ over $k$, but it is not a $4$-basis. This is due, in part, to the fact that $R^2$ and $k$ are not linearly disjoint over $k^2$, a property that, under some conditions, is necessary for $\infty$-integrability. In this paper, we will prove the following two results:

	\medskip 
\noindent{\bf Theorem \ref{main_1}} {\em Let $R$ be a reduced ring of prime characteristic $p>0$ that has a $p$-basis over a subfield $k$.   Then, if $R^p$ and $k$ are linearly disjoint over $k^{p}$, $\Der_{k}(R)=\IDer_{k}(R)$.
}

	\medskip

\noindent{\bf Theorem \ref{reciproco}}
{\em Let $R$ be an integral domain and let  $k\subset R$ be a field. Suppose that $R^p$ and $k$ are not linearly disjoint over $k^p$, that $\Omega_{R/k}$ is of finite presentation over $R$, and that $Q(R)$, the fraction field of $R$, is a finitely generated extension field of ${k}$.    Then $\IDer_{k}(R) \subsetneq \Der_k(R)$.}

\medskip
	
	It can be shown that if $R$ is a regular local ring containing a quasi-coefficent field $k$, then $R^p$ and $k$ are linearly disjoint over $k$. Hence, as a corollary to Theorem \ref{main_1}, we answer the question in \cite{FernandezNarvaez} under the assumption of the existence of a $p$-basis: 
	
	\medskip
	
	\noindent{\bf Corollary \ref{corolario_japones}}
	{\em Let $(R,{\mathfrak m})$ be a regular local ring containing a quasi-coefficient field $k$. Suppose that $R$ has a $p$-basis over $k$. Then $\Der_{k}(R)=\IDer_{k}(R)$. }
	
	\medskip
	
	Corollary \ref{corolario_japones} made us wonder what kinds of regular local rings that contain a quasi-coefficient field $k$ do not have a $p$-basis over $k$. During our search, we found the following result of Niitsuma, that is actually stronger than Corollary  \ref{corolario_japones}, and answers the question in \cite{FernandezNarvaez} using an approach that differs from ours (see also \cite[Theorem 2.1]{Tanimoto}):

	\begin{prop}\cite[Proposition 5.1]{NiitsumaTRU} Let $R$ be a Noetherian local ring of characteristic $p$ and let $k$ be a quasi-coeffcient field of $R$.  Then the following are equivalent: 
		
		1. $R$ is smooth over $k$; 
		
		2. $R$ is a regular local ring and $R$ has a $p$-basis over $k$.
		
		3. $R$ is a regular local ring and $R$ is a  finitely generated $R^p[k]$-module.	
	\end{prop}
	
	Example \ref{ejemplo_2} in section \ref{Sec-IntDis}  illustrates the case of  a local regular ring that contains a quasi-coefficient field $k$ that does not possess a $p$-basis over $k$.
	
	\medskip

	\medskip

The paper is organized as follows. Section \ref{seccion_preliminares} is devoted to recalling the main definitions regarding Hasse-Schmidt derivations and review some other basic results and definitions. The proofs of Theorem \ref{teo-Equivalencias} and its consequences will be addressed in section \ref{Sec-MainTheorem}. In section \ref{Sec-IntDis}, we will prove Theorems \ref{main_1} and \ref{reciproco}.

{\em Acknowledgments:} The author would like to thank Professor A. Bravo for proposing this problem and for her valuable assistance throughout the course of this work.
	
	\section{Preliminaries}\label{seccion_preliminares} 
	
	Let $A$ be a commutative ring with unit and let $k\subset A$ be a subring.  Let 
	$\overline{\mathbb N}:=\mathbb N \cup \{\infty\}$ and, for each
	integer $m\geq 1$,  we will use the notation  $A\lbr t\rbr_m:=A\lbr t\rbr/\langle
	t^{m+1}\rangle$ and $A\lbr t\rbr_\infty:=A\lbr t\rbr$. General
	references for the definitions and results in this section are
	\cite[\S 27]{Ma} and \cite{Na2}. The notion of HS-derivations (of length $\infty$) was introduced in \cite{H-S}:

	\begin{Def}\label{DefHS}
		A {\em Hasse-Schmidt derivation or  a   HS-derivation\footnote{The HS-derivations are also called {\em higher derivations}, see \cite[\S 27]{Ma}.} of $A$ over
			$k$  of length $m\geq 1$} (resp. of length $\infty$) is a sequence \color{black}
		$D:=(D_0,D_1,\ldots, D_m)$ (resp. $D=(D_0,D_1,\ldots)$) of
		$k$-linear maps $D_\alpha:A\rightarrow A$, satisfying the conditions:
		$$
		\begin{array}{ccc}
			D_0=\Id_A,&\displaystyle D_\alpha(xy)=\sum_{i+j=\alpha} D_i(x)D_j(y),
		\end{array}
		$$
		for all $x,y\in A$ and for all $\alpha$.  We write $\HS_k(A;m)$ (resp.
		$\HS_k(A;\infty)=\HS_k(A)$) for the set of HS-derivations of $A$ (over
		$k$) of length $m$ (resp. $\infty$).

	\end{Def}
	
	A   HS-derivation $D$ of $A$ over $k$ can also be seen as  a power series $\sum_{\alpha=0}^m D_\alpha t^\alpha \in \End_k(A)\lbr t\rbr_m$. In fact, $\HS_k(A;m)$ is a (multiplicative) sub-group   of $\mathcal U(\End_k(A)\lbr t\rbr_m)$, the group of units of $\End_k(A)\lbr t\rbr_m$. The group operation in $\HS_k(A;m)$ is   given by	$
	(D\circ D')_\alpha=\sum_{i+j=\alpha} D_i\circ D_j'$, and the identity element of $\HS_k(A;m)$ is $\mathbb{I}$ with $\mathbb I_0=\Id$ and $\mathbb I_\alpha=0$ for all $\alpha=1,\ldots, m$. Observe that   any HS-derivation $D\in \HS_k(A;m)$ determines and is determined by
	the $k$-algebra homomorphism
	$$
	\varphi_D: a\in A \longmapsto a+\sum_{\alpha\geq 1}^m D_\alpha(a)t^\alpha \in
	A\lbr t \rbr_m.
	$$
	
	Given positive integers  $1\leq n \leq m$, there is a  group   homomorphism $\tau_{m,n}: \HS_k(A;m) \to \HS_k(A;n)$ corresponding to the obvious truncation  map. 
	We have the following identity of groups
	\begin{equation*} \label{eq:limit-finite}
		\HS_k(A;\infty)
		= \lim_{\stackrel{\longleftarrow}{\substack{m}}} \HS(A;m).
	\end{equation*}

	\medskip
	
	If $J$ is an ideal of $A$, a $k$-derivation $\delta:A\to A$ is called {\em $J$-logarithmic} if $\delta(J) \subset J$. The set of
	$J$-logarithmic $k$-derivations is an $A$-submodule of $\Der_k(A)$
	denoted by $\Der_k(\log J)$. Analogously, we define $J$-logarithmic HS-derivations of length $m$, which is a subgroup of $\HS_k(A;m)$:
	\begin{Def}\label{Def-HSLogaritmica} Let $m\in \overline \NN$.
		We say that $D\in \HS_k(A;m)$ is {\em $J$-logarithmic} if $D_\alpha(J)\subseteq J$ for all $\alpha=0,\ldots, m$. The group of $J$-logarithmic HS-derivations of length $m$ is denoted by $\HS_k(\log J;m)$ and $\HS_k(\log J):=\HS_k(\log J;\infty)$.
	\end{Def}
	
	
	\begin{prop}\cite[Proposition 1.2.2]{TesisMP} 
		\label{Prop-HSsobreyectiva}
		If $A = k[x_i|\  i \in \mathcal I]/J$ for some ideal $J\subseteq k[x_i| \ i \in \mathcal I]$ then the map
		$$
		\begin{array}{rccc}
			\Pi_{\HS,m}:&
			\HS_k(\log J;m)&\to& \HS_k(R/J;m)\\
			&(D_i)_i&\mapsto& (\overline D_i) _i
		\end{array}
		$$
		where $\overline D_i (a+J)=D_i(a)+J$, is a surjective group homomorphism    for all $m \in \overline\N$.
	\end{prop}

		\begin{Def}\label{Log-IntHS} 
			
			For a given  $n\in \overline \N$, a $k$-derivation   $\delta:A\to A$   is {\em $n$-integrable} (over $k$) if there is a  HS-derivation $D\in \HS_k(A; n)$ such that $D_1=\delta$.
			Any such $D$
			is an  {\em $n$-integral} of $\delta$. The set of $n$-integrable $k$-derivations of $A$ is denoted by $\IDer_k(A; n)$ and $\IDer_k(A):=\IDer_k(A;\infty)$. 
			
			More generally, a HS-derivation $D\in \HS_k(A;m)$ is {\em $n$-integrable} (over $k$), where $m\leq n$, if there is a HS-derivation $E\in \HS_k(A;n)$ such that $\tau_{n,m}(E)=D$. Any such $E$ will be called an $n$-integral of $D$ and the set of $n$-integrable HS-derivation of $A$ (over $k$) of length $m$ is denoted by $\IHS_k(A;m;n)$.

			Let $J$ be an ideal of $A$. We say that $\delta$ is {\em $J$-logarithmically $n$-integrable} if there exists $D\in \HS_k(\log J;n)$ such that $D$ is an $n$-integral of
			$\delta$. We use $\IDer_k(\log J;n)$ to denote the set of $J$-logarithmically
			$n$-integrable derivations and $\IDer_k(\log
			J):=\IDer_k(\log J; \infty)$.
		\end{Def}

	The sets $\IDer_k(A;m)$ and $\IDer_k(\log J;m)$ are $A$-submodules of $\Der_k(A)$. We have the following chains:
	\begin{equation}
		\label{chain1}
		\Der_k(A)=\IDer_k(A;1)\supseteq \IDer_k(A;2)\supseteq \ldots \supseteq \IDer_k(A),
	\end{equation}
	and 
	\begin{equation*}
		\label{chain2}
		\Der_k(\log J)=\IDer_k(\log J;1)\supseteq \IDer_k(\log J;2)\supseteq \ldots \supseteq \IDer_k(\log J).  
	\end{equation*}

	As a consequence of Proposition \ref{Prop-HSsobreyectiva}, if $A$ is a $k$-algebra, i.e., if $A=k[x_i| \ x_i\in \mathcal I]/J$, there exist surjective maps of $A$-modules (see \cite[Corolario 1.2.3]{TesisMP}): 
	$$
	\IDer_k(\log J;m)\ni\delta\longrightarrow \overline\delta\in \IDer_k(A;m) \ \  \mbox{where $\overline \delta(r+J)=\delta(r)+J$, \ } \forall m\in \overline\NN.
	$$


	\color{black}
	
	If $\QQ\subset k$, then any derivation of $A$ is $\infty$-integrable,   and hence  $\Der_k(A)=\IDer_k(A)$ (see \cite[p.230]{mat-intder-I}). However, some of the  containments in sequence  (\ref{chain1})  could be strict   when the characteristic is positive (see \cite[Examples 1 to 3]{mat-intder-I}) and in this case we say that $A$ has a leap. Namely: 
	\begin{Def}\label{Def-Leap}
		 We say that an integer $s>1$ is a leap of $A$ over $k$ if $\IDer_k(A;s-1)\neq \IDer_k(A;s)$. The set of leaps of $A$ over $k$ is denoted by $\Leap_k(A)$.
	\end{Def}

	We recall the following results that we will use in the next section. 
	
		\begin{lem}\label{Lem-isom-isom}\cite[Lemma 1.1.5]{TesisMP}
		Suppose $f:A\to B$ is an isomorphism of $k$-algebras. Then  the map 
		$$\begin{array}{ccc}
			\HS_k(A;m) &\to & \HS_k(B;m)\\
			(D_r)_r &\mapsto &D_f =(f\circ D_r\circ f^{-1})
		\end{array}
		$$
		is a group isomorphism.
	\end{lem}

		\begin{teo}\cite[Teorema 27.1]{Ma}\label{Teo-Smooth}
				If $A$ is $0$-smooth over $k$, then a HS-derivation of $A$ over $k$ of  length $m$ is $\infty$-integrable. 
			\end{teo}
	\medskip

\section{Main Theorem}\label{Sec-MainTheorem}

In this section, we establish the main result of the present paper (see Theorem \ref{teo-Equivalencias}). Specifically, we show that, under the assumption of the existence of $p$-basis, $p^e$-integrability  for some positive integer $e$ is equivalent to the existence of $p^{e+1}$-basis. We recall the definition of this latter notion and some properties.

\begin{Def}\cite[Definition 3.1.1]{FurNii2002}\label{pebase}
	Let $R$ be a ring of prime characteristic $p$ and let $k\subset R$ be a subring.  We will say that a  subset $B\subset R$ is  {\em $p^e$-independent over $k$} if the set $B_e:=\{b_1^{\alpha_1}\cdots b_m^{\alpha_m} \ | \ b_i\in B, \ b_i\neq b_j, \ 0\leq \alpha_i<p^e\}$ is linearly independent over   $R^{p^e}[k]$. 
	We will say that $B$ is a  {\em $p^e$-basis of  $R$ over  $k$} if  $B$ is  $p^e$-independent over  $k$ and  $R=R^{p^e}[k][B]$. 
\end{Def}

Recall that if $B=\{b_i\ | \ i\in \mathcal I\}$ is a subset of $R$, $B$ is a $p^e$-basis of $R$ over $k$ if and only if $B^{p^t}=\{b_i^{p^t} \ | \ i \in \mathcal I\}$ is a $p$-basis of $R^{p^t}[k]$ over $R^{p^{t+1}[k]}$ for all $0\leq t<e$. Moreover, if $R$ is reduced and $R$ has a $p$-basis over $R^p$, then $R$ has a $p^e$-basis over $R^{p^e}$.  We also recall the following result that is an immediately consequence of \cite[(3.1.2) Lemma]{FurNii2002}:

\begin{lem}\label{Lem-Isomorfismo} 
	Let $R$ be a ring of prime characteristic $p$, let $k\subset R$ be a subring and let $B\subset R$ be a $p^e$-basis over $k$. Then there is an isomorphism of  rings  
	$$
	R\simeq R_{e,k}:=R^{p^e}[k][T_b\ | \ b\in B]/\langle T_b^{p^e}-b^{p^e}\rangle.
	$$
	that is also an isomorphism of  $R^{p^e}[k]$-algebras.
\end{lem}

\begin{lem}\label{lem-IntegrabilidaddeRe}
	Let  $L$ be a ring of prime characteristic  $p>0$ and let  $e\geq 1$ be an integer. Let  $A\subset L$ be a subset  and define: 
	$$
	S:=L[T_a\ | \ a\in A]/\langle T_a^{p^e}-a \ | \ a\in A\rangle. 
	$$
	Then, 
	$$
	\HS_{L}(S;m)=\IHS_{L}(S;m;p^e-1)
	$$
	for all $1\leq m\leq p^e-1$. In particular, $\Der_{L}(S)=\IDer_{L}(S;p^{e}-1)$.
\end{lem}

\begin{proof}
	Let $D\in \HS_L(S;m)$, let $I=\langle T_a^{p^e}-a \ | \  a\in A\rangle\subset  L[T_a\ | \ a\in A]$,  and consider the group homomorphism 
	$$
	\begin{array}{rccc}
		\Pi_{\HS,m}:&
		\HS_L(\log I;m)&\to& \HS_L(S;m)\\
		&(D_i)_i&\mapsto& (\overline D_i) _i
	\end{array}
	$$
	as in 	Proposition \ref{Prop-HSsobreyectiva}.  By  Proposition \ref{Prop-HSsobreyectiva}, there is some  $E\in \HS_L(\log I;m)$ whose image under   $\Pi_{\HS,m}$ is $D$. By Theorem   \ref{Teo-Smooth},  $E$ is  $\infty$-integrable over $L$. Let $\widetilde E$ be an $\infty$-integral of  $E$. Then    $\widetilde E_i(T_{{a}}^{p^e}-a)=\widetilde E_i(T_{{a}}^{p^e})=0$ for $i<p^e$,   since  $p^e$ does not divide $i$. Hence, $E':=\tau_{\infty,p^e-1}(\widetilde E)\in \HS_L(\log I; p^e-1)$. Therefore,  $\Pi_{\HS,p^e-1}(E')\in \HS_L(S;p^e{-1})$ is a $(p^e-1)$-integral of $D$. 
\end{proof}

\begin{teo}\label{teo-Equivalencias}
	Let $R$ be a ring of prime characteristic $p>0$ and $e\geq 1$ be a positive integer. Assume that $R$ has a $p$-basis over a subring $k$. Then the following properties are equivalent:
	\begin{enumerate}
		\item[1.]  $R$ has a $p^e$-basis over $k$.	
		\item[2.] $R$ has a $p^s$-basis over $k$ for all $s \leq e$.
		
		\item[3.] $\Der_k(R) = \IDer_k(R; p^{e}-1)$.
		
		\item[4.] $\HS_k(R; m) = \IHS_k(R; m; p^{e}-1)$ for all $m \leq p^e-1$.
	\end{enumerate}
\end{teo}	
	\begin{proof}
		It is clear that $2\Rightarrow 1$ and $4 \Rightarrow 3$. We prove that $1\Rightarrow 4$. First observe that for  
		$m<p^e$,  $\HS_{k}(R;m)=\HS_{R^{p^e}[k]}(R;m)$, hence, it suffices to show that 
		$\HS_{R^{p^e}[k]}(R;m)=\IHS_{R^{p^e}[k]}(R;m;p^e-1)$.   
		By Lemma \ref{Lem-Isomorfismo}, since $R$ has a $p^e$-basis,  there is an isomorphism   of $R^{p^e}[k]$-algebras,      $R\simeq R_{e,k}$. Now this implication follows from Lemmas \ref{lem-IntegrabilidaddeRe}  and \ref{Lem-isom-isom}.
		
		It remains to prove that $2\Rightarrow 3.$ For this, we follow the ideas of  the proof of \cite[\S 2 Lemma]{Fur81}. Let $B\subset R$ be a $p$-basis of $R$ over $k$. We will show that $B$ is a $p^s$-basis for all $1 \leq s \leq e$, and we proceed by induction on $s$. For $s = 1$, the result is clear. Suppose that $B$ is a $p^t$-basis for all $1 \leq t < s\leq e$, and we prove it for $s$.
		
		Since $B$ is a $p^{s-1}$-basis, we have that $B^{p^{t}}=\{b^{p^{t}}\mid b\in B\}$ is a $p$-basis of $R^{p^t}[k]$ over $R^{p^{t+1}}[k]$ for all $1\leq t\leq s-2$. Therefore, it remains to show that $B^{p^{s-1}}$ is a $p$-basis of $R^{p^{s-1}}[k]$ over $R^{p^s}[k]$. It is clear that $R^{p^{s-1}}[k]=R^{p^s}[k,B]$, so we need to prove that $B^{p^{s-1}}$ is $p$-independent over $R^{p^s}[k]$. 
		
		Let $C$ be a maximal subset of $B^{p^{s-1}}$ such that $C$ is $p$-independent over $R^{p^s}[k]$. If $C\neq B^{p^{s-1}}$, then there exists $a\in B^{p^{s-1}}\setminus C$ such that $C\cup \{a\}$ is not $p$-independent over $R^{p^{s-1}}[k]$. Hence, we have a relation: 
		\begin{equation}\label{EcTeorema}
				c_0+c_1a+\ldots +c_ra^r=0
		\end{equation}
of minimal degree with $c_r\in R^{p^{s-1}}[k,C]\setminus 0$ and $1\leq r\leq p-1$. Since $a\in B^{p^{s-1}}$, there is $b\in B$ such that $a=b^{p^{s-1}}$. Then, by \cite[Proposition 1.]{Fur81},  there exists $\delta \in \Der_k(R)$ such that $\delta(b)=1$ and $\delta(x)=0$ for all $x\in B\setminus \{b\}$. By hypothesis, $\Der_k(R)=\IDer_k(R;p^{e}-1)$, so let us consider $D\in \HS_k(R;p^{e}-1)$ an $(p^{e}-1)$-integral of $\delta$. Observe that $D_i$ es $R^{p^{s}}[k][C]$-lineal for all $1\leq i\leq p^{s-1}$ because $D_{p^{s-1}}(x^{p^{s-1}})=D_1(x)^{p^{s-1}}=0$ for all $x\in B\setminus \{b\}$.  Now, we have
		$$
		D_{p^s}(c_0+c_1a+\ldots+c_ra^r)=c_1D_{1}(b)^{p^{s-1}}+\ldots c_r D_1(b)^{p^{s-1}}=c_1+2c_2a+ \ldots+rc_ra^{r-1}=0
		$$
		Since $rc_r \neq 0$, this contradicts the minimality of the degree in (\ref{EcTeorema}).
		\end{proof}

\begin{obs}\label{Nota-SobreTeoCasoNopbase}
Note that the previous theorem does not hold if we remove the condition on the existence of a $p$-basis. For example, consider $k=\FF_2$ and $R=\FF_2[x,y]/\langle x^3-y^3\rangle$. It is easy to see that $R$ does not have a $p$-basis over $k$; however, $\Der_k(R)=\IDer_k(R)$ (see \cite[Proposition 4.1]{TiBinomial}). 
\end{obs}

The following corollaries are immediate consequences of the previous result:
\begin{cor}\label{Cor-PrimerSalto}
	Let $R$ be a ring of characteristic $p$ containing a subring $k$. Assume that $R$ has a $p^e$-basis over $k$ for some $e\geq 1$ but does not have a $p^{e+1}$-basis over $k$. Then, $p^e\in \Leap_k(R)$ and $p^s\not\in \Leap_k(R)$ for all $1\leq s<e$. 
\end{cor}

\begin{cor}\label{Cor-InfinitoIntegrabilidad}
	Let $R$ be a ring of prime characteristic $p>0$ which contains a subring   $k$.  Asume that $R$ has a $p$-basis over $k$. Then, the following conditions are equivalent:
	\begin{enumerate}
		\item[1.]  $R$ has a $p^e$-basis over $k$ for all $e\geq 1$.
		\item[2.] $\Der_k(R)=\IDer_k(R)$. 
		\item[3.] $\HS_k(R;m)=\IHS_k(R;m)$ for all $m\geq 1$.
	\end{enumerate}
\end{cor}

\begin{proof} It is clear that $3\Rightarrow 2$. The implication $2 \Rightarrow 1$ follows from Lemma \cite[\S 2 Lemma]{Fur81}. Finally, we prove $1\Rightarrow 3$. Let $D\in \HS_{k}(R;m)$ and select  $e\in {\mathbb N}$   so that $p^{e-1}\leq m<p^e$. Since $R$ has a $p^e$-basis  over $k$, by Theorem  \ref{teo-Equivalencias} it follows that  $\HS_{k}(R;m)=\IHS_{k}(R;m;p^e-1)$. Now, let 
	$D^{\{e\}}\in \HS_{k}(R;p^e-1)$ be a $(p^e-1)$-integral of $D$. Again, since $R$ has a $p^{e+1}$-basis over $k$,  it follows that  $\HS_{k}(R;p^e-1)=\IHS_{k}(R;p^e-1;p^{e+1}-1)$, and hence there is some $(p^{e+1}-1)$-integral    $D^{\{e+1\}}\in \HS_{k}(R;p^{e+1}-1)$ of  $D^{\{e\}}$. Iterating this process, for all  $\epsilon\geq e$, we can find  $D^{\{\epsilon\}}\in \HS_{k}(R;p^\epsilon-1)$ so that  $\tau_{p^\epsilon-1, p^{\epsilon-1}-1}(D^{\{\epsilon\}})=D^{\{\epsilon-1\}}$. The inverse limit of these derivations is an $\infty$-integral of $D$ and we conclude the proof. 
\end{proof}

\section{Integrability and linearly disjointness}\label{Sec-IntDis}

In this section, we present two results relating integrability and linearly disjointness, which allow us to partially answer the question in \cite[3.16 Remark]{FernandezNarvaez}. We recall that if $K$ is a field and $K\subset R, S \subset \Omega$ are two ring extensions, we say that $R$ and $S$ are {\em linearly disjoint} over $K$ if the canonical map $R\otimes_K S\to \Omega$ is inyective. The first of our result is a generalization of a theorem due to A. Tyc in the case of noetherian rings (\cite[Theorem 2]{Tyc}) and, in the second one we will use two well-known properties of field extensions that, for the convenience of the reader, we recall here:

\begin{lem}\cite[27.G Exercise 1]{MatsumuraViejo}\label{Lema-MacLane} Let 
	$k\subseteq K$ be a field extension of positive characteristic   $p$. Then  $K$ is  separable over $k$ if and only if  $K$ and $k^{1/p}$ are linearly disjoint over  $k$. 
\end{lem}

\begin{lem}\cite[\S2 Corollary]{Fur81}\label{Lem-CorolarioSeparebleCuerpos} Let $K$ be a finitely generated field over a field $k$. Then the following are equivalent: 
	\begin{enumerate}
		\item[1.] $K$ is separable over $k$. 
		\item[2.] $\Der_k(K)=\IDer_k(K)$.
	\end{enumerate}
\end{lem}

\begin{teo} \label{main_1}
	Let $R$ be a reduced ring of prime characteristic $p>0$ that has a $p$-basis $B$  over a subfield $k$.  Then, if $R^p$ and $k$ are linearly disjoint over $k^{p}$, $B$ is a  $p^e$-basis of $R$ over $k$ for all $e\geq 1$ and therefore $\Der_{k}(R)=\IDer_{k}(R)$.
\end{teo}

\begin{proof}
	First of all it can be checked, using an inductive argument, that because $R$ is reduced and $R^p$ and $k$ are linearly disjoint over $k^{{p}}$, then $R^{p^e}$ and $k^{p^{e-1}}$ are linearly disjoint over $k^{p^e}$ for all $e\in {\mathbb N}_{\geq 2}$. Next,   we will prove that for $e\geq 1$, $B$  is a $p^e$-basis of $R$ over $k$.     To do so we will first show, using induction, that the following property holds for all $e\geq 1$: 
	
	\medskip
	
	{\bf Property *.} {\em Let $R$, $k$ and  $B$ be as in the theorem, and let   $A$ be a $p$-basis of  $k/k^p$. Then, for all  $e\geq 1$,  $R$ is free over  $R^{p^e}[k^{p^{e-1}}]$ with basis $B_eA_{e-1}$ (see Definition \ref{pebase}).}
	
	\medskip
	
	\noindent The base case, $e=1$ is inmediate since by assumption,   $R$ is free over $R^p[k]$ with $p$-basis $B$.   Suppose the property holds for some $e\geq 1$.  Since $R^{p^e}[k^{p^{e-1}}]\simeq R^{p^e}\otimes_{k^{p^e}}k^{p^{e-1}}$, $R$ is free over $R^{p^e}$ with basis $B_eA_e$. Now, $R^{p^e}$ is free over $R^{p^{e+1}}[k^{p^e}]$ with $p$-basis $B^{p^e}=\{b^{p^e} \ | \ b\in B\}$, thus $R$ is free over $R^{p^{e+1}}[k^p]$ with basis $B_{e+1}A_e$. 
	
	\medskip
	
	To conclude, using that the property * holds,  we have that  $R$ is free over $R^{p^{e+1}}[A_e]$ with basis $B_{e+1}$, but 
	$R^{p^{e+1}}[A_e]=R^{p^{e+1}}[k]$.

	%
	%
	%
		%
		%
		%
		%
		%
		%
		%
\end{proof}

\begin{obs}\label{Nota-LinealmenteDisjuntos} The previous result is not true if we remove the hypothesis on the existence of $p$-basis. For example, let $k$ be a perfect field of characteristic $2$, and $R=k[x,y]/\langle x^2-y^3\rangle$. In this case, $k$ and $R^p$ are linearly disjoint over $k^p=k$, however $\IDer_k(R)\neq \Der_k(R)$ (see \cite[Proposition 4.1]{TiBinomial}). 
\end{obs}

\begin{teo}\label{reciproco}
	Let $R$ be an integral domain and let  $k\subset R$ be a field of positive characteristic $p$. Suppose that $R^p$ and $k$ are not linearly disjoint over $k^p$, that $\Omega_{R/k}$ is of finite presentation over $R$, and that $Q(R)$, the quotient field of $R$, is a finitely generated extension field of ${k}$.   \ Then $\IDer_{k}(R) \subsetneq \Der_{k}(R)$.
\end{teo}

\begin{proof} Suppose, to get to a contradiction, that $\text{IDer}_{k}(R)= \text{Der}_{k}(R)$.   Since  
	$$\Hom_R(\Omega_{R/{k}}, R)\otimes_RQ(R)\simeq  \Hom_{Q(R)}(\Omega_{Q(R)/{k}}, Q(R)),$$ if $\IDer_{k}(R)= \Der_{k}(R)$, then $\IDer_{k}(Q(R))= \Der_{k}(Q(R))$ (see \cite[1.2.10]{Na2}). Since $R^p$ and $k$ are not linearly disjoint over ${k}^p$, then 
	$Q(R)^p$ and ${k}$ are not linearly disjoint over ${k}^p$, hence $Q(R)$ is not separable over ${k}$ (see Lemma  \ref{Lema-MacLane}).  By Lemma \ref{Lem-CorolarioSeparebleCuerpos}, we obtain a contradiction, which completes the proof.
\end{proof}

As an immediate consequence of the previous results, we obtain the following corollary. 
\begin{cor}
	Let $R$ be an integral domain and let $k\subset R$ be a field of positive characteristic $p$. Suppose that $R$ has a $p$-basis over $k$, that $\Omega_{R/k}$ is of finite presentation over $R$, and that $Q(R)$ is a finitely generated extension field of $k$. Then the following properties are equivalent: 
\begin{enumerate}
	\item[1.] $R^p$ and $k$ are linearly disjoint over $k^p$. 
	\item[2.] $\Der_k(R)=\IDer_k(R)$. 
\end{enumerate}
\end{cor}

The following result provides an affirmative answer to the question in \cite[3.16 Remark]{FernandezNarvaez} in the case where $R$ has a $p$-basis over a quasi-coefficient field $k$. Note that this corollary can be proved using \cite[Proposition 5.1]{NiitsumaTRU}  and Theorem \ref{Teo-Smooth}. Nevertheless, we give an alternative proof based on linearly disjointness.

\begin{cor}\label{corolario_japones}
	Let $(R,{\mathfrak m}, K)$ be a regular local ring containing a quasi-coefficient field $k$. Suppose that $R$ has a $p$-basis over $k$. Then $\Der_{k}(R)=\IDer_{k}(R)$. 
\end{cor}

\begin{proof}
	It suffices to check that  $R^p$ and $k$ are linearly disjoint over $k^p$. To this end we will  show that there is a basis $B$ of $k/k^p$  that is linearly independent over $R^p$. Let $A$ be a $p$-basis of $k/k^p$. Then,  since   $K/k$ is formally \'etale,   $A$ is also a $p$-basis of ${K}/{K}^p$. To conclude, by Lemma \ref{lema_2_5_KimNit} below, $A$ is a set of independent elements over $R^p$. 
\end{proof}

\begin{lem}\cite[Lemma 2.5]{KimNit_Japan} \label{lema_2_5_KimNit}
	Let $(R,{\mathfrak m},  {K})$ be a local domain ring with quotient field $Q(R)$. Let $R^p\subset R'\subset R$ be an intermediate regular local ring, $(R', {\mathfrak m}', K')$. Let $A$ be a subset of $R$ such that $\overline{A}$ is $p$-independent over $Q(R')$. Then $A$ is $p$-independent over $Q(R')$.  
	
\end{lem}

The following example illustrates that there are regular local rings containing a quasi-coefficient field $k$ that do not have a $p$-basis over $k$. However, we still do not know if there exist regular local rings containing a quasi-coefficient field $k$ such that $\Der_k(R)\neq \IDer_{k}(R)$.

\begin{ej}\label{ejemplo_2}
	In 	\cite[Section 4.1]{D-S} it is given an example of a non-excellent local regular ring. To facilitate the reading of this paper, we briefly describe the construction of such ring. Let $k$ be an algebraic closure of    $\FF_p$ and let $q(t)\in k\lbr t\rbr$ be a non-unit that is not  algebraic over $k(t)$. With these assumptions the following ring homomorphism is injective: 
	$$
	\begin{array}{ccc}
		k[x,y]&\hookrightarrow&k\lbr t\rbr\\
		x&\mapsto &t\\
		y&\mapsto & q(t),
	\end{array}
	$$
	that permits to see an isomorphic copy of $k(x,y)$ in $k((t))$.  The $t$-adic valuation in $k((t))$ gives as a discrete valuation ring $V_q=k\lbr t\rbr \cap k(x,y)$, with maximal ideal  $\langle x\rangle$. Observe that the residue field of $V_q$ is $k$. This means that the given valuation is not divisorial and then $V_q$ is not excellent (see \cite[Theorem 4.1]{D-S}). 
	
	Observe that $V_q$ contains a coefficient field, $k$. We will check next that $V_q$ cannot have a $p$-basis over $k$. Suppose, on the contrary, that $V_q$ has  a $p$-basis over $k$. Then, $\Omega_{V_q/k}$ has a differential basis.    But $$
	\Omega_{V_q/\FF_p}\simeq \Omega_{V_q/k}
	$$
	and therefore,    $\Omega_{V_q/\FF_p}$  also has  a differential basis. By   \cite[Theorem 1]{Tyc} it follows that  $V_q$ has an absolute  $p$-basis, but then     $V_q$ would be excellent (see \cite{KimuraNiitsumaTRU}). 
\end{ej}

\end{document}